\documentclass[12pt]{article}
\usepackage{pifont}
\usepackage{amsmath}
\allowdisplaybreaks[4]
\usepackage{amssymb,float}
\usepackage{amsthm,graphicx,pifont}
\usepackage{tikz}
\usetikzlibrary{trees}
\usepackage{graphicx,psfrag,float,subfigure}
\usepackage{caption,enumerate}
\usepackage[scale=0.8,a4paper]{geometry}
\usepackage[colorlinks]{hyperref}
\usepackage{amsthm,amsmath,amssymb}
\usepackage{mathrsfs}
\newtheorem{theorem}{Theorem}[section]

\newtheorem{lemma}[theorem]{Lemma}
\newtheorem{proposition}[theorem]{Proposition}

\title{ Double Italian domination in trees}
\author{Weiping Shang$^{1*}$, \ {Shanshan Zhang$^{2}$\thanks{ e-mail: shangwp@zzu.edu.cn, zhangss@jsnu.edu.cn.}}\\
{\small $^{1}$  School of Mathematics and Statistics, Zhengzhou University,
 Zhengzhou 450001, China}\\
{\small $^{2}$  School of Mathematics and Statistics, Jiangsu Normal University,
Xuzhou 221116, China}\\
}
\date{}\makeatother
\begin{document}
\maketitle

\begin{abstract}
Let $G$ be a graph with vertex set $V=V(G)$. A double Roman dominating function on a graph $G$ is a function $f : V \to \{0,1,2,3\}$ satisfying the conditions that if $f(v) = 0$, then vertex $v$ must have at least two neighbors in $V_2$ or one neighbor in $V_3$; if $f(v) = 1$, then vertex $v$ must have at least one neighbor in $V_2 \cup V_3$. The weight of a double Roman dominating function $f$ is the sum $f(V) = \sum_{v \in V} f(v)$, and the double Roman domination number $\gamma_{dR}(G)$ is the minimum weight of a double Roman dominating function on $G$. A double Italian dominating function on a graph $G$ is a function $f : V \to \{0,1,2,3\}$ satisfying the condition that for every vertex $u \in V$, if $f(u) \in \{0,1\}$, then $\sum_{v \in N[u]} f(v) \ge 3$. The double Roman domination number $\gamma_{dI}(G)$ is the minimum weight of a double Italian dominating function on $G$. Mojdeh and Volkmann [D.A. Mojdeh and L. Volkmann, Roman {3}-domination (double Italian domination), Discrete Appl. Math. 283 (2020), 555--564] proved that $\gamma_{dI}(T) = \gamma_{dR}(T)$ for any tree $T$. However, we find that there is a minor issue in the proof. In this paper, we first prove that $\gamma_{dI}(T) \neq \gamma_{dR}(T)$. Subsequently, we present a sharp bound on the double Italian domination number of any non-trivial tree $T$, and characterize the trees attaining this bound.
\end{abstract}

{\bf Keywords} domination, double Roman domination number, double Italian domination number

{\bf 2020 MR Subject Classification} \ 05C70

\section{Introduction}

In this paper, we only consider simple undirected graphs. For standard graph-theoretical notions and terminology, we refer the reader to \cite{Bondy2008}. Let $G$ be a graph with vertex set $V = V(G)$ and edge set $E = E(G)$. The \textit{order} of $G$ is the number of its vertices. The open neighborhood of a vertex $v \in V$ is the set $N(v)=\{u \in V \mid uv \in E\}$, and its closed neighborhood is $N[v] = N(v) \cup \{v\}$. Vertices in $N(v)$ are called the \textit{neighbors} of $v$. The \textit{degree} of vertex $v \in V$ is $d(v) = |N(v)|$. A vertex with degree one is called a \textit{leaf}, and its neighbor is a \textit{support vertex}. A support vertex with two or more leaf neighbors is called a \textit{strong support vertex}. An edge adjacent to a leaf is called a \textit{pendant edge}.

A graph is \textit{trivial} if it has only one vertex, and \textit{non-trivial} otherwise. A tree is an acyclic connected graph. For any $S \subseteq V$, we denote by $G - S$ the graph obtained from $G$ by deleting all the vertices in $S$ together with all edges incident with the vertices in $S$, and by $G[S]$ the subgraph induced by $S$.

A set $S \subseteq V$ in a graph $G$ is called a \textit{dominating set} if every vertex of $G$ is either in $S$ or adjacent to a vertex of $S$. The \textit{domination number} $\gamma(G)$ equals the minimum cardinality of a dominating set in $G$, and a dominating set of $G$ with cardinality $\gamma(G)$ is called a $\gamma$-set of $G$.

Given a graph $G$ and a positive integer $k \geq 2$, assume that $f: V \to \{0, 1, 2, \dots, k\}$ is a function, and suppose that $(V_0, V_1, V_2, \dots, V_k)$ is the ordered partition of $V$ induced by $f$, where $V_i = \{v \in V: f(v) = i\}$ for $i \in \{0, 1, \dots, k\}$. So we can write $f = (V_0, V_1, V_2, \dots, V_k)$.

A \textit{Roman dominating function} on a graph $G$ is a function $f : V \to \{0,1,2\}$ satisfying the condition that every vertex $u$ for which $f(u) = 0$ is adjacent to at least one vertex $v$ for which $f(v) = 2$. The \textit{weight} of a Roman dominating function is the sum $f(V) = \sum_{v \in V} f(v)$. The minimum weight of a Roman dominating function on $G$ is called the \textit{Roman domination number} of $G$ and is denoted $\gamma_R(G)$. A Roman dominating function on $G$ with weight $\gamma_R(G)$ is called a $\gamma_R$-function of $G$.

The original study of Roman domination was motivated by the defense strategies used to defend the Roman Empire during the reign of Emperor Constantine the Great. He decreed that for all cities in the Roman Empire, at most two legions should be stationed. Further, if a location having no legions was attacked, then it must be within the vicinity of at least one legion being stationed, and so that one of the two legions could be sent to defend the attacked city. This part of the history of the Roman Empire gave rise to the mathematical concept of Roman domination, originally defined and discussed by Stewart \cite{Stewart1999} in 1999, and subsequently developed by Cockayne et al. \cite{Cockayne2004} in 2004. Since then, a lot of related variations and generalizations have been studied.

Beeler et al. \cite{Beeler2016} introduced the concept of double Roman domination. What they propose is a stronger version of Roman domination that doubles the protection by ensuring that any attack can be defended by at least two legions.

A \textit{double Roman dominating function} on a graph $G$ is a function $f : V \to \{0,1,2,3\}$ satisfying the following conditions:
\begin{enumerate}
    \item[(i)] If $f(v) = 0$, then vertex $v$ must have at least two neighbors in $V_2$ or one neighbor in $V_3$.
    \item[(ii)] If $f(v) = 1$, then vertex $v$ must have at least one neighbor in $V_2 \cup V_3$.
\end{enumerate}
The weight of a double Roman dominating function $f$ on $G$ is the sum $f(V) = \sum_{v \in V} f(v)$, and the minimum weight of $f$ over every double Roman dominating function on $G$ is called the \textit{double Roman domination number} of $G$. We denote this number with $\gamma_{dR}(G)$, and a double Roman dominating function of $G$ with weight $\gamma_{dR}(G)$ is a $\gamma_{dR}$-function of $G$.

Mojdeh and Volkmann\cite{Mojdeh2020} defined a variant of double Roman domination, namely double Italian domination (Roman 3-domination). Formally, a \textit{ double Italian dominating function} on a graph $G$ is a function $f : V \to \{0,1,2,3\}$ satisfying the condition that for every vertex $u \in V$, if $f(u) \in \{0,1\}$, then $\sum_{v \in N[u]} f(v)\ge 3$. Specifically, it satisfies the following conditions:
\begin{enumerate}
    \item[(i)] If $f(v) = 0$, then vertex $v$ must have at least three neighbors in $V_1$, or one neighbor in $V_1$ and one neighbor in $V_2$, or two neighbors in $V_2$, or one neighbor in $V_3$.
    \item[(ii)] If $f(v) = 1$, then vertex $v$ must have at least two neighbors in $V_1$, or one neighbor in $V_2 \cup V_3$.
\end{enumerate}
The weight of a double Italian dominating function $f$ on $G$ is the sum $f(V)= \sum_{v \in V} f(v)$, and the minimum weight of $f$ over every double Italian dominating function on $G$ is called the \textit{double Italian domination number} of $G$. We denote this number with $\gamma_{dI}(G)$, and a double Italian dominating function of $G$ with weight $\gamma_{dI}(G)$ is a $\gamma_{dI}$-function of $G$.

Compared with the requirements of a double Roman dominating function, a double Italian dominating function relaxes the restrictions on the domination method and is more flexible. It is easy to see that every double Roman dominating function is a double Italian dominating function, so $\gamma_{dI}(G) \leq \gamma_{dR}(G)$.
For any non-trivial tree $T$, Mojdeh and Volkmann\cite{Mojdeh2020} proved that $\gamma_{dI}(T) = \gamma_{dR}(T)$ for any tree $T$. However, we find that there is a minor issue in the proof.

In this paper, we first prove that $\gamma_{dI}(T) \neq \gamma_{dR}(T)$. Subsequently, we present a sharp bound on the double Italian domination number of any non-trivial tree $T$, and give a characterization of trees with $\gamma_{dI}(T) = 2\gamma(T) + 1$.

 \section{Preliminaries}
In this section, we first present some preliminary results, which are used in the proof of our main results.

\begin{proposition}[\cite{Beeler2016}]\label{pro1}
In a double Roman dominating function of weight $\gamma_{dR}(G)$, no vertex needs to be assigned the value 1.
\end{proposition}

By Proposition \ref{pro1}, we can assume that $V_1 = \emptyset$ for all double Roman dominating functions under consideration.

\begin{proposition}[\cite{Beeler2016}]\label{pro2}
For any graph $G$, $2\gamma(G) \leq \gamma_{dR}(G) \leq 3\gamma(G)$.
\end{proposition}

\begin{proposition}[\cite{Beeler2016}]\label{pro3}
If $T$ is a non-trivial tree, then $\gamma_{dR}(T) \geq 2\gamma(T) + 1$.
\end{proposition}

For a positive integer $t$, a \textit{wounded spider} is obtained from a star $K_{1,t}$ by subdividing at most $t-1$ edges. Similarly, for an integer $t\ge 2$, a \textit{healthy spider} is obtained from a star $K_{1,t}$ by subdividing all of its edges. In a wounded spider, a vertex of degree $t$ is called the \textit{head vertex}, and the vertex that is distance two from the head vertex is the \textit{foot vertex}. The neighbor of a foot vertex is called the \textit{subdivision vertex}. The head and foot vertices are well defined except when the wounded spider is the path on two or four vertices. For $P_2$, we will consider both vertices to be head vertices, and in the case of $P_4$, we will consider both end vertices as foot vertices and both interior vertices as head vertices.

\begin{proposition}[\cite{Ahangar2017, Zhang2018}]\label{pro4}
If $T$ is a non-trivial tree, then $\gamma_{dR}(T) = 2\gamma(T) + 1$ if and only if $T$ is a wounded spider.
\end{proposition}

\begin{theorem}[\cite{Mojdeh2020}]\label{th5}
 For any tree $T$, $\gamma_{dI}(T) = \gamma_{dR}(T)$.
\end{theorem}

Since $\gamma_{dI}(G) \le \gamma_{dR}(G)$, if there exists a $\gamma_{dI}$-function on $G$ that is also a double Roman dominating function, then this function is also a $\gamma_{dR}$-function and $\gamma_{dR}(G) = \gamma_{dI}(G)$. Before the discussion, we first present an important lemma, which is frequently used in the following section.
\begin{lemma}\label{le6}
There exist a $\gamma_{dI}$-function and a $\gamma_{dR}$-function on $G$ such that every leaf and its unique neighbor are assigned values $2$ and $0$, or $0$ and $3$, respectively, and every strong support vertex is assigned a value 3.
\end{lemma}
\noindent\textit{Proof.} Let $f$ be a $\gamma_{dI}$-function, $v$ be a leaf of $G$, and $u$ be its unique neighbor. By the definition of $f$, if $f(v) = 0$, then $f(u) = 3$. If $f(v) = 1$, then $f(u) = 2$. Redefine a function $g$ such that $g(v) = 0$, $g(u) = 3$, and $g(w)=f(w)$ for all other vertices. Clearly, $g$ is also a double Italian dominating function with the same weight. If $f(v) = 3$, then $f(u) = 0$. Similarly, redefine a function $g$ such that $g(v) = 0$, $g(u) = 3$, and $g(w)=f(w)$ for all other vertices. Clearly, $g$ is also a double Italian dominating function with the same weight. If $f(v) = 2$ and $f(u) = 1$, then replace them with $0$ and $3$ as above. Hence if $f$ is a $\gamma_{dI}$-function, then $f(v) = 2$ and $f(u) = 0$ or $f(v) = 0$ and $f(u) = 3$.

By Proposition \ref{pro1}, let $f$ be a $\gamma_{dR}(G)$-function with no vertex assigned value $1$. Let $v$ be a leaf vertex of $G$ and $u$ be its unique neighbor. By the definition of $f$, if $f(v) = 0$, then $f(u) = 3$ and if $f(v) = 2$, then $f(u) = 0$. If $f(v) = 3$, then $f(u) = 0$. Similarly, redefine a function $g$ such that $g(v) = 0$, $g(u) = 3$, and $g(w)=f(w)$ for all other vertices. Clearly, $g$ is also a double Italian dominating function with the same weight. Hence if $f$ is a $\gamma_{dR}$-function, then $f(v) = 2$ and $f(u) = 0$ or $f(v) = 0$ and $f(u) = 3$.

Let $w$ be a strong support vertex  of $G$. Since $w$ has at least two leaf neighbors, if any double Italian dominating function or double Roman dominating function of $T$ assigns a value less than $3$ to $w$, then the total weight of assigned to $w$ and its leaf neighbors is at least $4$. This implies that any $\gamma_{dI}$-function or $\gamma_{dR}$-function of $T$ will assign a $3$ to $u$ and $0$ to its leaf neighbors.$\square$ \\

\section{Main results}

In this section we present a tree for which the two domination numbers are not equal. The \textit{double star} $S_{p,q}$ is the tree with exactly two adjacent non-leaf vertices (central vertices), one of which is adjacent to $p$ leaves and the other to $q$ leaves. Let $u_1$ and $u_2$ be the two central vertices of the double star $S_{2,2}$, and $T_1$ be the tree obtained by subdividing the three pendant edges of the double star $S_{2,2}$, as shown in the following figure.

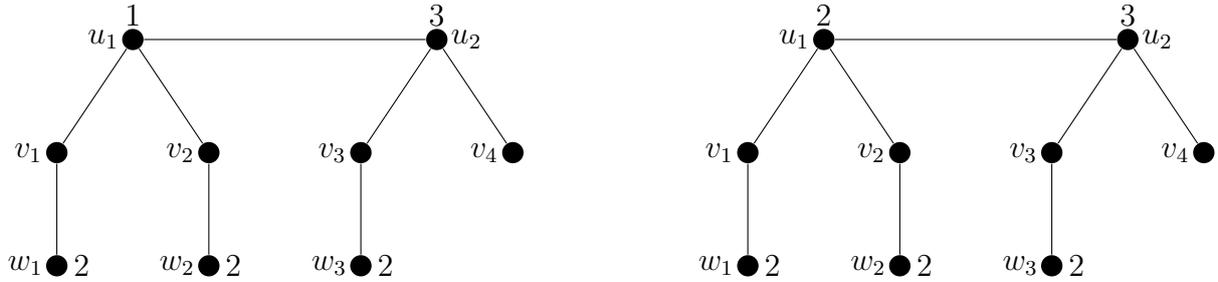
\begin{figure}[h]
\centering
\begin{tikzpicture}[level distance=1.5cm,
                    level 1/.style={sibling distance=2cm},
                    level 2/.style={sibling distance=1cm},
                    every node/.style={circle,fill=black,minimum size=8pt,inner sep=0pt}]
  \node[label=above: 1,label=left:$u_1$](n1) {}
    child { node [label=left:$v_1$] {}
      child { node[label=left:$w_1$, label=right:2 ] {} }
    }
    child { node[label=left:$v_2$] {}
      child { node[label=left:$w_2$,label=right:2] {} }
    };
  \node[label=above:3,label=right:$u_2$] (n3) at (4,0) {}
    child { node [label=left:$v_3$]{}
      child { node[label=left:$w_3$,label=right:2] {} }
    }
    child { node[label=left:$v_4$] {}
    };
    \draw (n1) -- (n3);
\end{tikzpicture}
\hspace{2cm}
\begin{tikzpicture}[level distance=1.5cm,
                    level 1/.style={sibling distance=2cm},
                    level 2/.style={sibling distance=1cm},
                    every node/.style={circle,fill=black,minimum size=8pt,inner sep=0pt}]
  \node[label=above: 2,label=left:$u_1$](u) {}
    child { node [label=left:$v_1$] {}
      child { node[label=left:$w_1$, label=right:2 ] {} }
    }
    child { node[label=left:$v_2$] {}
      child { node[label=left:$w_2$,label=right:2] {} }
    };
  \node[label=above:3,label=right:$u_2$] (w) at (4,0) {}
    child { node [label=left:$v_3$]{}
      child { node[label=left:$w_3$,label=right:2] {} }
    }
    child { node[label=left:$v_4$] {}
    };
  \draw (u) -- (w);
\end{tikzpicture}
\caption{$\gamma_{dI}(T_1) = 10$  and  $\gamma_{dR}(T_1) = 11$ }
\end{figure}

\begin{theorem}\label{main th1}
There exists a tree $T_1$ such that $\gamma_{dI}(T_1) < \gamma_{dR}(T_1)$.
\end{theorem}
\noindent\textit{Proof.} Recall that $u_1$ and $u_2$ are the two central vertices of the double star $S_{2,2}$, and $T_1$ be the tree obtained from a double star $S_{2,2}$ by subdividing two pendant edges incident to $u_1$, and one pendant edge incident to $u_2$. Let $v_1$ and $v_2$ be the other two neighbors of $u_1$, $v_3$ and $v_4$ be the other two neighbors of $u_2$, and $w_i$ be the leaf adjacent to $v_i$, for $i = 1, 2, 3$ in $T_1$. Since one pendant edge incident to $u_2$ is not subdivided, $v_4$ is also a leaf.

By Lemma \ref{le6}, there exists a $\gamma_{dI}$-function $f$ such that either $f(w_i) = 2$ and $f(v_i) = 0$, or $f(w_i) = 0$ and $f(v_i) = 3$, for $i = 1, 2, 3$, and either $f(v_4) = 2$ and $f(u_2) = 0$, or $f(v_4) = 0$ and $f(u_2) = 3$.

If $f(w_3) = 2$ and $f(v_3) = 0$, then by the definition of a double Italian dominating function, $f(w_3)+f(u_2) \ge 3$, which implies $f(u_2) \ge 1$. From the above discussion, we have $f(u_2) = 3$ and $f(v_4) = 0$, so $f(u_2) + f(v_3) + f(w_3) + f(v_4) = 5$. If $f(w_3) = 0$ and $f(v_3) = 3$, then $f(u_2) + f(v_3) + f(w_3) + f(v_4) \ge 5$. Similarly, we obtain $f(u_1) + f(v_1) + f(w_1) \ge 3$ and $f(v_2) + f(w_2) \ge 2$. Hence $\gamma_{dI}(T_1) \ge 10$.

On the other hand, let $f(w_i) = 2$ for $i = 1, 2, 3$, $f(v_i) = 0$ for $i = 1, 2, 3, 4$, $f(u_1) = 1$, and $f(u_2) = 3$. It is straightforward to verify that this is a double Italian dominating function with weight 10. Thus $\gamma_{dI}(T_1) = 10$.

Similarly, by Lemma \ref{le6}, there exists a $\gamma_{dR}(T_1)$-function $g$ such that either $g(w_i) = 2$ and $g(v_i) = 0$, or $g(w_i) = 0$ and $g(v_i) = 3$, for $i = 1, 2, 3$, and either $g(v_4) = 2$ and $g(u_2) = 0$, or $g(v_4) = 0$ and $g(u_2) = 3$.

If $g(w_3) = 2$ and $g(v_3) = 0$, then by the definition of a double Roman dominating function, $g(u_2) \ge 2$. From the above discussion, we have $g(u_2) = 3$ and $g(v_4) = 0$, so $g(u_2) + g(v_3) + g(w_3) + g(v_4) = 5$. If $g(w_3) = 0$ and $g(v_3) = 3$, then $g(u_2) + g(v_3) + g(w_3) + g(v_4) \ge 5$. Similarly, if $g(w_i) = 2$ and $g(v_i) = 0$, for $i = 1$ or $2$, then by the definition of a double Roman dominating function, $g(u_1) \ge 2$, and so $g(u_1) + g(v_1) + g(w_1) + g(v_2) + g(w_2) \ge 6$. If $g(w_i) = 0$ and $g(v_i) = 3$ for $i = 1,2$, then $g(u_1) + g(v_1) + g(w_1) + g(v_2) + g(w_2) \ge 6$. Hence $\gamma_{dR}(T_1) \ge 11$.

On the other hand, let $g(w_i) = 2$ for $i = 1, 2, 3$, $g(v_i) = 0$ for $i = 1, 2, 3, 4$, $g(u_1) = 2$, and $g(u_2) = 3$. It is straightforward to verify that this is a double Roman dominating function with weight 11. Thus $\gamma_{dR}(T_1) = 11$ and $\gamma_{dI}(T_1) < \gamma_{dR}(T_1)$. $\square$ \\

By theorem \ref{main th1}, we conclude that the statement $\gamma_{dI}(T) = \gamma_{dR}(T)$ for any tree $T$ is incorrect.

Let $u_1$ and $u_2$ be the two central vertices of the double star $S_{p,q}$, where $p\ge 2$ and $q\ge 2$. Let $T$ be the tree obtained from a double star $S_{p,q}$ by subdividing all pendant edges incident to $u_1$, and subdividing at most $q$ pendant edges incident to $u_2$. By a similar argument, we have $\gamma_{dI}(T) < \gamma_{dR}(T)$.

We now present a lower bound on the double Italian domination number of any non-trivial tree T.

\begin{theorem}\label{main th2}
If $T$ is a non-trivial tree, then $\gamma_{dI}(T) \ge 2\gamma(T)+1$.
\end{theorem}
\noindent\textit{Proof.} Let $T$ be a non-trivial tree of order $n$. We proceed by induction on $n$. Note that if $T$ is the star $K_{1,n-1}$, then $\gamma(K_{1,n-1})=1$ and $\gamma_{dI}(K_{1,n-1})=3=2\gamma(K_{1,n-1})+1$. If $\text{diam}(T)=3$, then $T$ is a double star $S_{p,q}$ (where $1\le p\le q$), and $\gamma(S_{p,q})=2$ and $\gamma_{dI}(S_{p,q})=5$ when $p=1$ and $\gamma_{dI}(S_{p,q})=6$ when $p\ge 2$. We have $\gamma_{dI}(S_{1,q})=2\gamma(S_{1,q})+1$ and $\gamma_{dI}(S_{p,q})> 2\gamma(S_{p,q})+1$ when $p\ge 2$. Hence, we may assume that the diameter of $T$ is at least $4$. This implies that $n\ge 5$.

Assume that any tree $T'$ with order $2\le n'<n$ has $\gamma_{dI}(T')\ge 2\gamma(T')+1$.

We first show that the result holds if $T$ has a strong support vertex. Assume that $T$ has a strong support vertex $u$ adjacent to a leaf $v$. Since $u$ is a strong support vertex, any $\gamma_{dI}$-function of $T$ will assign a $3$ to $u$ and $0$ to its leaf neighbors. Consider $T'=T-v$. Since $u$ is a support in $T'$, we have $\gamma(T')=\gamma(T)$ and $\gamma_{dI}(T') \le \gamma_{dI}(T)$. Moreover, $T'$ is a non-trivial tree, so we can apply our inductive hypothesis showing that $\gamma_{dI}(T)\ge\gamma_{dI}(T')\ge 2\gamma(T')+1=2\gamma(T)+1$. Since the desired result holds for graphs having a strong support vertex, henceforth we may assume that every support vertex is adjacent to exactly one leaf.

Among all leaves of $T$, choose $r$ and $u$ to be two leaves such that the distance between $r$ and $u$ is the diameter of $T$. We root the tree $T$ at vertex $r$. Let $w$ be the unique neighbor of $u$, and $x$ be the parent of $w$. Note that by our choice of $u$, every child of $w$ is a leaf of $T$. Since $T$ has no strong support vertices, it follows that $d(w)=2$.

Among all $\gamma_{dI}$-functions of $T$, let $f$ be one such that $f(u)+f(w)$ is minimized. By lemma \ref{le6}, we have either $f(u)=2$ and $f(w)=0$ or $f(u)=0$ and $f(w)=3$. We consider the following two cases:

\noindent\textbf{Case 1.} $f(u)=2$ and $f(w)=0$.

Let $T'=T-\{u,w\}$, and $f'$ be the restriction of $f$ onto $T'$. Note that $f'$ is a double Italian dominating function of $T'$. Then $\gamma_{dI}(T')\le \gamma_{dI}(T)-2$. Moreover, any dominating set of $T'$ can be extended to a dominating set of $T$ by adding $w$, so $\gamma(T)\le \gamma(T')+1$. By inductive hypothesis, we have $\gamma_{dI}(T)\ge \gamma_{dI}(T')+2\ge 2\gamma(T')+1+2\ge2(\gamma(T)-1)+3=2\gamma(T)+1$, as desired.

\noindent\textbf{Case 2.} $f(u)=0$ and $f(w)=3$.

We claim that $f(x)=0$ and $d(x)=2$. Suppose that $f(x)\ge 1$. Let $g$ be the function such that $g(u)=2$, $g(w)=0$, $g(x)=\hbox{min}\{f(x)+1,3\}$, and $g(v)=f(v)$ for all other vertices. Note that $g$ is a function with weight no more than $f$, and $g(u)+g(w)<f(u)+f(w)$, contradicting our choice of $f$. Thus, $f(x)=0$.

Suppose that $d(x)\ge 3$. As we have established, every child of $x$ is a leaf or a support of a leaf. Since $f(x)=0$, any leaf neighbor of $x$ must be assigned $2$, and every support vertex adjacent to $x$ is assigned $3$ under $f$. If $x$ has one leaf neighbor $z$, then let $g$ be the function such that $g(x)=3$, $g(z)=g(w)=0$, $g(u)=2$ and $g(v)=f(v)$ for all other vertices. Note that $g$ is a function with the same weight as $f$, and $g(u)+g(w)<f(u)+f(w)$, contradicting our choice of $f$.
If $x$ has only two or more support vertices as children, let $g$ be the function obtained assigning $2$ to $x$, $0$ to each of the children of $x$, $2$ to their leaf neighbors, and $g(v)=f(v)$ for all other vertices, then $g$ is a function with weight no more than $f$, and $g(u)+g(w)<f(u)+f(w)$, contradicting our choice of $f$. Thus, $d(x)=2$.

 Consider $T'=T-\{u,w,x\}$. Since $\text{diam}(T)\ge 4$, $T'$ is a non-trivial tree. Further, since any $\gamma$-set of $T'$ can be extended to a dominating set of $T$ by adding the vertex $w$, we have $\gamma(T)\le \gamma(T')+1$. Let $f'$ be the restriction of $f$ onto $T'$. Recall that $f(w)=3$ and $f(u)=f(x)=0$, then $f'$ is a double Italian dominating function of $T'$. Thus, $\gamma_{dI}(T')\le \gamma_{dI}(T)-3$. By inductive hypothesis, we have $\gamma_{dI}(T)\ge \gamma_{dI}(T')+3\ge 2\gamma(T')+1+3 \ge 2(\gamma(T)-1)+4=2\gamma(T)+2$. The result follows.$\square$ \\

We next characterize the class of trees $T$ for which $\gamma_{dI}(T)=2\gamma(T)+1$.

\begin{theorem}\label{main th3}
If $T$ is a non-trivial tree, then $\gamma_{dI}(T) =2\gamma(T)+1$ if and only if $T$ is a wounded spider.
\end{theorem}

\noindent\textit{Proof.} Let $T$ be a wounded spider and $u$ be the head vertex. Let $S=\{v:d(u,v)=2\}$ be the set of foot vertices. Since every leaf or its support vertex must be in any dominating set, it is clearly that $S\cup \{u\}$ forms a $\gamma$-set for $T$. Also, if $V_0=V-S-\{u\}$, $V_2=S$, and $V_3=\{u\}$, then $f=(V_0,V_2,V_3)$ is an double Italian dominating function with $f(V)=2\gamma(T)+1$. Therefore, $f$ is a $\gamma_{dI}$-function.

Now we will prove that if $\gamma_{dI}(T)=2\gamma(T)+1$, then $T$ is a wounded spider. Suppose to the contrary, and let $T$ be a smallest counterexample, we have $\gamma_{dI}(T)=2\gamma(T)+1$ and $T$ is not a wounded spider. By the proof of Theorem \ref{main th2}, if $T$ is the star $K_{1,n-1}$, then $\gamma_{dI}(K_{1,n-1})=2\gamma(K_{1,n-1})+1$. If $\text{diam}(T)=3$, then $T$ is the double star $S_{p,q}$ (where $1\le p\le q$). We have $\gamma_{dI}(S_{1,q})=2\gamma(S_{1,q})+1$ and $\gamma_{dI}(S_{p,q})> 2\gamma(S_{p,q})+1$ when $p\ge 2$. $K_{1,n-1}$ and $S_{1,q}$ are wounded spiders.
 Hence, we assume that the diameter of $T$ is at least $4$.

Assume that $T$ has a strong support vertex $u$ adjacent to a leaf $v$. Since $u$ has at least two leaf neighbors, any $\gamma_{dI}$-function of $T$ will assign a $3$ to $u$ and $0$ to its leaf neighbors. Consider $T'=T-v$. Since $u$ is a support in $T'$, we have $\gamma(T')=\gamma(T)$ and $\gamma_{dI}(T') \le \gamma_{dI}(T)$. Moreover, $T'$ is a non-trivial tree, by theorem \ref{main th2}, $\gamma_{dI}(T)\ge\gamma_{dI}(T')\ge 2\gamma(T')+1=2\gamma(T)+1$. Note that $\gamma_{dR}(T)=2\gamma(T)+1$, which implies $\gamma_{dR}(T')=2\gamma(T')+1$. By the minimality of $T$, $T'$ ia a wounded spider. If $u$ is the center of $T'$, then $T$ is also a wounded spider, contradicting the assumption that $T$ is a counterexample. Hence, $u$ is a subdivision vertex of $T'$ and the degree of the center of $T'$ is at least 3. Note that every vertex of $T$ is either a leaf or a support vertex. By lemma \ref{le6}, there exists a $\gamma_{dI}$-function $f$ of $T$, no vertex is assigned the value 1. Note that $f$ is also a $\gamma_{dR}$-function of $T$. Then $\gamma_{dI}(T)=\gamma_{dR}(T)$. By proposition \ref{pro3} and \ref{pro4}, we have $\gamma_{dI}(T)>2\gamma(T)+1$, a contradiction. Hence, we may assume that every support vertex is adjacent to exactly one leaf.

Among all leaves of $T$, choose $r$ and $u$ to be two leaves such that the distance between $r$ and $u$ is the diameter of $T$. We root the tree $T$ at vertex $r$. Let $w$ be the unique neighbor of $u$, and $x$ be the parent of $w$. Note that by our choice of $u$, every child of $w$ is a leaf of $T$. Since $T$ has no strong support vertices, it follows that $d(w)=2$.

Among all $\gamma_{dI}$-functions of $T$, let $f$ be one such that $f(u)+f(w)$ is minimized. By lemma \ref{le6}, we have either $f(u)=2$ and $f(w)=0$ or $f(u)=0$ and $f(w)=3$.
By the case 2 of theorem \ref{main th2}, if $f(u)=0$ and $f(w)=3$, then $\gamma_{dI}(T)\ge 2\gamma(T)+2$, contradicting that $\gamma_{dI}(T)=2\gamma(T)+1$. Hence, $f(u)=2$ and $f(w)=0$.

Let $T'=T-\{u,w\}$, and $f'$ be the restriction of $f$ onto $T'$. Note that $f'$ is a double Italian dominating function of $T'$. Then $\gamma_{dI}(T')\le \gamma_{dI}(T)-2$. Moreover, any dominating set of $T'$ can be extended to a dominating set of $T$ by adding $w$, so $\gamma(T)\le \gamma(T')+1$. By theorem \ref{main th2}, we have $\gamma_{dI}(T)\ge \gamma_{dI}(T')+2\ge 2\gamma(T')+1+2\ge2(\gamma(T)-1)+3=2\gamma(T)+1$. Note that $\gamma_{dI}(T)=2\gamma(T)+1$, which implies that $\gamma_{dI}(T')=2\gamma(T')+1$ and $\gamma(T)=\gamma(T')+1$. By the minimality of $T$, $T'$ is a wounded spider. If $x$ is the head vertex of $T'$, then $T$ is also a wounded spider, contradicting the assumption that $T$ is a counterexample.

If $x$ is a subdivision vertex of $T'$ and the degree of head vertex of $T'$ is at least 3,
then every vertex of $T$ is either a leaf or a support vertex. By lemma \ref{le6}, there exists a $\gamma_{dI}$-function $f$ of $T$, no vertex is assigned the value 1. Note that $f$ is also a $\gamma_{dR}$-function of $T$, and $\gamma_{dI}(T)=\gamma_{dR}(T)$. By proposition \ref{pro3} and \ref{pro4}, we have $\gamma_{dI}(T)>2\gamma(T)+1$, a contradiction.

If $x$ is a foot vertex of $T'$, then let $S$ be the set of foot vertices of $T'$, and $y$ be the head vertex of $T'$. Since every leaf or its support vertex must be in any dominating set, we have $D'=S\cup \{y\}$ forms a $\gamma$-set for $T'$, where $x\in D'$. Let $D=D'\cup \{w\}-\{x\}$. Then $D$ is a $\gamma$-set for $T$. Thus, $\gamma(T)=\gamma(T')$, contradicting that $\gamma(T)=\gamma(T')+1$.

If $x$ is a leaf neighbor of the head vertex $y$, and $y$ has only one leaf neighbor.
Note that $\gamma(T)=\gamma(T')$, contradicting that $\gamma(T)=\gamma(T')+1$. Hence, $x$ is a leaf neighbor of $y$, and $y$ has at least two leaf neighbors. By lemma \ref{le6}, we have either $f(y)=0$ or $f(y)=3$ in $T$. If $f(y)=0$, then any leaf neighbor of $y$ must be assigned $2$. Recall $f(u)=2$ and $f(w)=0$, implying that $f(x)=2$. Let $g$ be the function obtained assigning $3$ to $y$, $0$ to each of the leaf neighbor of $y$, $1$ to $x$, and $g(v)=f(v)$ for all other vertices. Then $g$ is a function with weight no more than $f$ and $g(y)=3$. Without loss of generality, let $f(y)=3$. Note that $f(u)=2$ and $f(w)=0$, we have $f(x)=1$.

Consider $T''=T-\{u,w,x\}$. Since $\text{diam}(T)\ge 4$, $T''$ is a non-trivial tree. Further, since any $\gamma$-set of $T''$ can be extended to a dominating set of $T$ by adding the vertex $w$, $\gamma(T)\le \gamma(T'')+1$. Let $f''$ be the restriction of $f$ onto $T''$. Recall $f(y)=3$, then $f''$ is a double Italian dominating function of $T''$. Thus, $\gamma_{dI}(T'')\le \gamma_{dI}(T)-3$. By theorem \ref{main th2}, we have $\gamma_{dI}(T)\ge \gamma_{dI}(T'')+3\ge 2\gamma(T'')+1+3 \ge 2(\gamma(T)-1)+4=2\gamma(T)+2$, a contradiction.

Therefore, $T$ is a wounded spider. The result follows. $\square$ \\

\section*{Acknowledgements}

This work is  supported by the National Natural Science Foundation of China under grant number 12571381.


\begin{thebibliography}{1}

\bibitem{Ahangar2017}
H.A. Ahangar, M. Chellali, S.M. Sheikholeslami, On the double Roman domination in graphs, Discrete Appl. Math. 232 (2017) 1--7.

\bibitem{Bondy2008}
J.A. Bondy, U.S.R. Murty, Graph Theory, Springer-Verlag, Berlin, 2008.

\bibitem{Beeler2016}
R.A. Beeler, T.W. Haynes, S.T. Hedetniemi, Double Roman domination, Discrete Appl. Math. 211 (2016) 23--29.

\bibitem{Chellali2016}
M. Chellali, T.W. Haynes, S.T. Hedetniemi, A. MacRae, Roman {2}-domination, Discrete Appl. Math. 204 (2016) 22--28.

\bibitem{Cockayne2004}
E.J. Cockayne, P.A. Dreyer, S.M. Hedetniemi, S.T. Hedetniemi, Roman domination in graphs, Discrete Math. 278 (2004) 11--22.

\bibitem{Henning2003}
M.A. Henning, S.T. Hedetniemi, Defending the roman empire--a new strategy, Discrete Math. 266 (2003) 239--251.

\bibitem{Haynes1998}
T.W. Haynes, S.T. Hedetniemi, P.J. Slater, Fundamentals of Domination in Graphs, Marcel Dekker, Inc., New York, 1998.

\bibitem{Kloster2019}
W. Klostermeyer, G. MacGillivray, Roman, Italian, and $2$-domination, J. Combin. Math. Combin. Comput. 108 (2019) 125--146.

\bibitem{Stewart1999}
I. Stewart, Defend the Roman empire!, Sci. Amer. 281 (6) (1999) 136--139.

\bibitem{Mojdeh2020}
D.A. Mojdeh, L. Volkmann, Roman {3}-domination (double Italian domination), Discrete Appl. Math. 283 (2020) 555--564.

\bibitem{Zhang2018}
X.J. Zhang, Z.P. Li, H.Q. Jiang, Z.H. Shao, Double Roman domination in trees, Inf. Process. Lett. 134 (2018) 31--34.

\end{thebibliography}
\end{document}